\documentclass{amsart}
\usepackage{amsfonts,amssymb,amsmath,euscript}
\usepackage{graphicx}
\usepackage[matrix,arrow,curve]{xy}
\usepackage{wrapfig}

  \usepackage{amssymb}
\usepackage{mathrsfs}
\usepackage{euscript}
\usepackage{lipsum}

    \newtheorem{thm}{Theorem}                    
    \newtheorem{thm*}{Theorem}

    \newtheorem{lemma*}{Lemma}    

    \newtheorem{conj}{Conjecture}                 

\newcommand*{\abs}[1]{\left\lvert#1\right\rvert}   
\newcommand*{\norm}[1]{\left\Vert#1\right\Vert}    

\newcommand{\hilb}{\mathcal H}
\newcommand{\scR}{\mathscr R}
\newcommand{\Scal}[1]{\left\langle #1\right\rangle}               
\newcommand{\mb}{\mathbb}
\newcommand{\cl}{\mathcal}

\newcommand{\mbC}{\mathbb C}

\begin{document}
\title{Spectral flow inside essential spectrum V: \\ on absorbing points of coupling resonances}
\author{Nurulla Azamov}
\address{Independent scholar, Adelaide, SA, Australia}
\email{azamovnurulla@gmail.com}
\begin{abstract} 
Let $H_0$ and $V$ be self-adjoint operators, such that $V$ admits a factorisation $V = F^*JF$ with bounded self-adjoint $J$ and $|H_0|^{1/2}$-compact $F.$
Coupling resonance functions, $r_j(z),$ of the pair $H_0$ and $V$ can be defined as $r_j(z) = - \sigma_j(z) ^{-1},$ where $\sigma_j(z)$ are eigenvalues of the compact-operator
valued holomorphic function $F(H_0-z)^{-1} F^*J.$ Taken together, the functions $r_j(z)$ form an infinite-valued holomorphic function on the resolvent set of~$H_0.$
These functions contain a lot of information about the pair $H_0, V$ (this is well-known in the case of rank one $V$).
A point $z_0$ of the resolvent set we call \emph{absorbing} if some $r_j(z)$ goes to $\infty$ as $z \to z_0$ along some half-interval.

In this note I present some partial results concerning absorbing points of coupling resonances. 
\end{abstract}
\maketitle

\bigskip 
Let $H_0$ and $V$ be self-adjoint operators on a (complex separable) Hilbert space $\hilb,$ such that $V$ admits a factorisation $V = F^*JF$ with bounded self-adjoint $J$ and $|H_0|^{1/2}$-compact $F.$
\emph{Coupling resonance functions}, $r_j(z),$ of the pair $H_0$ and $V$ can be defined as $$r_j(z) = - \sigma_j(z) ^{-1},$$ where $\sigma_j(z)$ are eigenvalues of the compact-operator
valued holomorphic function $F(H_0-z)^{-1} F^*J.$ Taken together, the functions $r_j(z)$ form an infinite-valued holomorphic function, which we denote $\scR(z) = \scR(z; H_0,V),$ on the resolvent set of~$H_0.$
These functions contain a lot of information about the pair $H_0$ and $V,$  --- this is well-known in the case of a rank one perturbation $V = \Scal{v, \cdot} v,$ which in essence is equivalent to the theory of Herglotz-Nevanlinna-Pick functions,
see e.g. \cite{Do,SiTrId2}, since in this case the function
$$
    \mb C _+ \  \ni  \ z \quad \mapsto \quad \Scal{v, (H_0 - z)^{-1} v}
$$
is Herglotz-Nevanlinna-Pick and it allows to recover the pair $H_0$ and $V$ up to a unitary equivalence, --- assuming the pair is irreducible, using the same unitary operator. 

 More information about coupling resonance functions, their applications and why they are so-called can be found in the introductions of papers \cite{AzSFIES, AzDaMN, AzDa2}.

The general case is complicated by the fact that the function $\scR(z)$ is infinite-valued, and can have potentially quite an erratic behaviour, typical for such functions. 
However, as it turns out, there are only two types of singularities given in the following theorem, see \cite[Theorem 3.1]{AzSFIESIII}:
\begin{thm}  \label{T1} Coupling resonance points can have only two types of singularities:
\begin{enumerate}
  \item continuous branching points of finite period, or
  \item absorbing points, whether isolated or not.
\end{enumerate}
Moreover, isolated absorbing points, if they exist, must have infinite period of branching. 
\end{thm}

Recall, \cite{AzSFIESIII}, that 
a point $z_0$ of the resolvent set is called \emph{absorbing} if some coupling resonance function goes to $\infty$ as $z \to z_0$ along some half-interval, $\gamma,$ ending at $z_0.$
In this case it is not difficult to see that the point $z_0$ has a small enough neighbourhood such that if another half-interval $\gamma'$ ending at $z_0$ is homotopic to $\gamma$ in the domain of holomorphy of $\scR(z)$
then the coupling resonance function goes to infinity along $\gamma'$ too. 

Clearly, Theorem \ref{T1} significantly reduces the possible types of singularities which coupling resonance functions can have.
I believe that more is true, as specified in the following
\begin{conj}  \label{Conj1} Assume that $H_0$ and $V$ is as above. Then the pair $H_0$ and $V$ does not have absorbing points, whether isolated or not. 
\end{conj}

It would be nice to have this proved, since then the analytic function $\scR(z)$ could have only one type of singularity: continuous branching points of finite period, which are quite manageable. 
Concerning these branching points, it is not difficult to demonstrate their existence. What is their meaning is currently being investigated, see \cite{AzDa2}. 

In this paper I present a solution to a very special case of Conjecture \ref{Conj1}. 

First I present some preliminaries. Let $s_1(T),  s_2(T), \ldots$ be $s$-numbers of a compact operator $T,$ see e.g. \cite{GK, SiTrId2} for details, and let $\lambda_1(T), \lambda_2(T), \ldots$ be eigenvalues of $T$ written in the order of decreasing magnitudes. 
The following holds for any $n=1,2,\ldots$ and $p\in(0,\infty),$ see \cite{GK}:
\begin{equation} \label{E1}
   \sum_{j=1}^n \abs{\lambda_j(T)}^p \leq    \sum_{j=1}^n [s_j(T)]^p.
\end{equation}
By $\norm{T}_1$ we denote the trace class norm of $T.$ We denote $R_z(H_0) = (H_0 - z)^{-1}.$

\begin{thm} \label{T: no isolated absorbing points for RzV from L1}
Let~$H_0$ and~$V$ be self-adjoint operators.
If~$V$ admits a factorisation $V = F^*JF$ with bounded self-adjoint~$J$ and closed $\abs{H_0}^{1/2}$-compact~$F$ 
such that $F R_z(H_0)F^*J$ is trace-class then all coupling resonance functions of the pair $(H_0,V)$ 
do not have isolated absorbing points. 
\end{thm}
\begin{proof} Assume that~$z_0$ is an isolated absorbing point of the necessarily  infinite-valued in a neighbourhood of $z_0$ (see Theorem \ref{T1}) resonance function 
with values $r_j(z), \ j \in \mb Z,$ in a neighbourhood of~$z_0.$ Without loss of generality, we can assume that 
$r_j$ take values in the upper half-plane. 
For a real $s,$ the numbers $(s-r_j(z))^{-1}$ are eigenvalues of the trace class operator 
$F R_z(H_s)F^*J$ and thus the series 
$$
  f(z) := \sum_{j \in \mb Z} (s-r_j(z))^{-1}
$$
converges absolutely and uniformly on compact subsets of a deleted neighbourhood of~$z_0.$ 
The function $f(z)$ is single-valued in the deleted neighbourhood and takes values in~$\mb C_+$
since all summands of the series take values in~$\mb C_+.$ Also, since~$z_0$ is an absorbing point,
the function~$f$ admits analytic continuation to~$z_0$ where it takes zero value. 
Indeed, by \eqref{E1}  for any $z \in O$ the series 
$$
      f(z) := \sum_{j \in \mb Z}  \abs{ s-r_j(z)}^{-1}
$$
is bounded by $\sup_{z \in O} \norm{F R_z(H_s)F^*J}_1,$ where $O$ is a neighbourhood of~$z_0$.
But this contradicts the openness principle, since $f\big|_O$ does not take values in~$\mb C_-.$
\end{proof}

The following theorem of T.\,Rad\'o
gives a positive answer to a very weak version of the famous Painlev\'e problem
about removable sets, see e.g. \cite{Dud}. 
Its proof can be found in \cite[\S\S 49]{Shab}. 
\begin{thm}\label{T: T. Rado} Let~$G$ be an open set in~$\mbC.$
If a continuous map $u \colon G \to \mbC$ is holomorphic outside of its 
set of zeros then it is holomorphic in~$G.$ 
\end{thm}
In my opinion, it seems to be highly unlikely that an absorbing point, if it exists, can be non-isolated. 
In any case, the set of absorbing points is small in the following sense: the set of rays emanating from  a point $z_0 \in \mbC_+$ which hit an absorbing point has Lebesgue measure zero.
This easily follows from the above mentioned properties of absorbing points, the Riemann Mapping Theorem and the Nevanlinna-Luzin-Privalov theorem. 
I present a proof of absence of non-isolated absorbing points in a very special case.
\begin{thm} \label{T: sexy prop, trace class case} Let $F R_z(H_0)F^*J$ be trace class, 
$O$ an open set in~$\mb C_+,$~$K$ a compact subset of~$O,$ and~$r(z)$
a possibly multi-valued coupling resonance function which is either regular in $O\setminus K$ 
or may have continuous branching points. If~$r(z)$ goes to $\infty,$ whenever~$z$ goes 
to a point of~$K$ along any half-interval in $O\setminus K,$ then the set~$K$ is empty. 
\end{thm}
\begin{proof} We define a function 
$$
  f(z) := \sum_{j} (s-r_j(z))^{-1},
$$
where the sum is taken over all branches of the multi-valued resonance function in $O\setminus K.$
By the argument of the proof of Theorem~\ref{T: no isolated absorbing points for RzV from L1},
the function $f(z)$ is continuous in~$O,$ 
is holomorphic outside~$K$ and vanishes in~$K.$
Hence, by Rado's Theorem~\ref{T: T. Rado}, $f(z)$ is holomorphic in~$O.$ So, the set of zeros 
of this function is discrete, and so is~$K.$ 
Hence, by Theorem~\ref{T: no isolated absorbing points for RzV from L1} the set $K$ is empty.
\end{proof}

It is possible that Conjecture \ref{Conj1} does not hold in full generality. If so, it would be interesting to find sufficient conditions on the pair $H_0$ and $V$ under which the conjecture holds.
In particular, the case of $F R_z(H_0) F^* \in \cl L_{1,\infty},$ where $\cl L_{1,\infty}$ is the Dixmier ideal, see e.g. \cite{CoNG}, seems to be especially interesting, see \cite{AMSZ}.

\bigskip 
{\it Acknowledgements.} I thank my wife for financially supporting me during the work on this paper.

\end{document}